\UseRawInputEncoding
\documentclass[11pt,twoside,reqno]{amsart}
\usepackage{url}
\usepackage{mathrsfs}
\usepackage{latexsym}
\usepackage{amsmath}
\usepackage{amssymb}
\usepackage{amsthm}
\usepackage{amscd}
\usepackage{amsfonts}
\usepackage{epsfig}
\usepackage{multicol}
\usepackage{CJK}
\usepackage{bbm}
\usepackage{amsmath, amsthm, amsfonts, amssymb,mathtools}
\usepackage{mathrsfs}

\def \Spec{\textup{Spec}}

\def \x{\overset{\sim}{x}}
\def \G{\overset{\sim}{G}}
\theoremstyle{plain}
\newtheorem{thm}{Theorem}[section]
\newtheorem{lem}[thm]{Lemma}

\newtheorem{rem}[thm]{Remark}

\newcommand\restr[2]{{
		\left.\kern-\nulldelimiterspace 
		#1 
		\right|_{#2} 
}}

\newcommand{\pf}{\noindent\begin {proof}}
\newcommand{\epf}{\end{proof}}
\makeatletter

\begin{document}
\title{A new approach for constructing graph being determined by their generalized $Q$-spectrum.}

\author{Liwen Gao}
\address{Liwen Gao: Department of Mathematics, Nanjing University,
	Nanjing 210093, China;  gaoliwen1206@smail.nju.edu.cn}

\author{Xuejun Guo$^{\ast}$}
\address{Xuejun Guo$^{\ast}$: Department of Mathematics, Nanjing University,
	Nanjing 210093, China;  guoxj@nju.edu.cn}
\address{$^{\ast}$Corresponding author}

\thanks{The authors are supported by National Nature Science Foundation of China (Nos. 11971226, 12231009).}

\date{}

\noindent

\begin{abstract}
	Given a graph $G$, we have the adjacency matrix $A(G)$ and degree diagonal matrix $D(G)$. The $Q$-spectrum is the all eigenvalues of $Q$-matrix $Q(G)=A(G)+D(G)$. A class of graphs is determined by their generalized $Q$-spectrum (DGQS for short) if any two graphs among the class have the same $Q$-spectrum and so do their complement imply that they are isomorphic. In \cite{tw}, the authors provides a new way to construct $DGQS$ graphs by considering the rooted product graphs $G\circ P_{k}$  and they prove when $k=2,3$, $G\circ P_{k}$ is $DGQS$ for a special graph $G$. In this paper, we will prove that under the same conditions for $G$, the conclusion is true for any positive integer $k$.
\end{abstract}
\medskip

\maketitle

\textbf{Keywords:} graph spectrum, rooted product graph, determined by their generalized $Q$-spectrum.

\vskip 10pt

\textbf{2020 Mathematics Subject Classification:} 05C50.

\section{ Introduction }
Throughout our paper, we only consider the simple and undirected graphs. Let $G$ be a graph with vertex set $V(G)=\{1,2,\cdots,n\}$ and we have the adjacency matrix $A(G)=(a_{ij})_{n\times n}$, where $a_{ij}=1$ if there is an edge between $i$ and $j$; $a_{ij}=0$, otherwise. The spectrums of a graph $G$ are all the eigenvalues of $A(G)$. ``Whether a graph is determined by its spectrum" is a typical question of the graph theory. It originated from the chemical theory in \cite{hp}. In 2006, Wang and Xu  extended the problem to the generalized spectrum in \cite{wx1,wx2}. The generalized spectrum is all the spectrums of $G$ and the complements of $G$. A graph $G$ is known as being determined by its generalized spectrum (DGS for short) if any graph $H$ has the same generalized spectrum as $G$ imply that $H$ and $G$ isomorphic.  Later in 2017, Wang give an effective algorithm of $DGS$ in \cite{w}: For a graph $G$,
its walk matrix $W(G)=[e,eA,\cdots,A^{n-1}e]$ satisfies the condition that $2^{-\lfloor\frac{n}{2}\rfloor}\det W$ is odd and square-free, then $G$ is DGS.

Given a graph $G$, we have degree diagonal matrix $D(G)$. The $Q$-spectrums are the all eigenvalues of $Q$-matrix $Q(G)=A(G)+D(G)$.  We denote all $Q$-spectrum of $G$ by $\Spec_Q(G)$ and the complement of $G$ by $\overline{G}$. A graph $G$ is said to be determined by the generalized $Q$-spectrum ($DGQS$ for short) if for any graph $H$ that satifies  $\Spec_Q(G)=\Spec_Q(H)$ and $\Spec_Q(\overline{G})=\Spec_Q(\overline{H})$ imply that $H$ is isomorphic to $G$. The study of $Q$-spectrum started from Cvetkovi\'{c} and Simi\'{c} in \cite{cs1,cs2,cs3}. Later, some scholars found some special graphs to be $DGQS$ in \cite{r,s}. Finally in \cite{QJW},  Qiu, Ji and Wang provided a general algorithm for testing whether a graph $G$ is determined by their generalized $Q$-spectrum: If its walk $Q$-matrix $W_Q(G)=[e,eQ,\cdots,Q^{n-1}e]$ satisfies the condition that $2^{-\lfloor\frac{n-3}{2}\rfloor}\det W_Q$ is odd and square-free, then $G$ is DGQS.

Constructing rooted product graph is a very interesting method to product a series of graphs with more vertices and edges from the original graph (see Godsil and McKay\cite{gm}). Given a graph $G$ of $n$ vertices and another graph $P_k$ of $k$ vertices, a rooted product graph $\G_k=G\circ P_k$ of the two graphs is generated from glueing one graph $G$ and $n$ copys $H$, i.e. glueing every vertices of $G$ and a copy of $H$ (see \ref{Figure 1.}). To be specific, we suppose the vertex set of $G$ and $P_k$ is $V(G)=\{1,2,\cdots,n\}$ and $V(P_k)=\{p_1,p_2,\cdots,p_k\}$ and denote the edge set of $G$ and $P_k$ by $E(G)$ and $E(P_k)$, respectively. Thus the vertex set of $\G_k$ is $V(\G_k)=\{(i,p_s)|1\le i\le n, 1\le s\le k\}$ and the edge set of $\G_k$ is $\{(i,p_1)\sim (j,p_1)|ij\in E(G)\}\cup\{(i,p_s)\sim (i,p_{s+1})|1\le i\le n, 1\le s\le k\}$.
\begin{figure}[h]
	\centering
	\includegraphics[width=10cm,height=6cm]{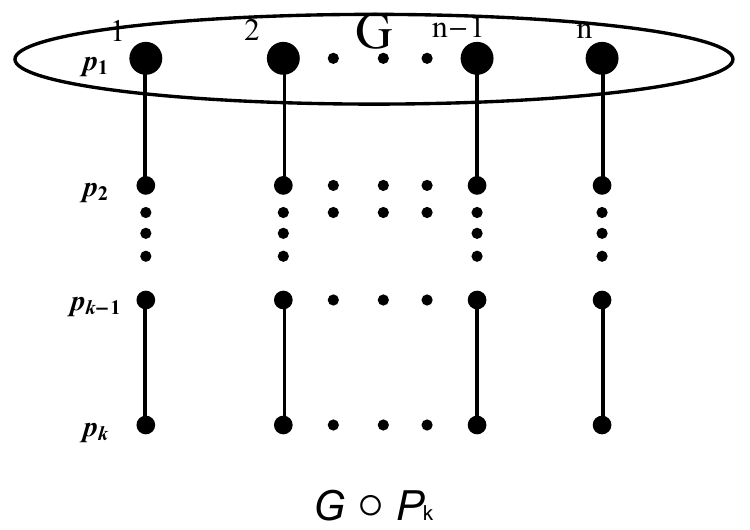}
	\caption{The rooted product graph $\widehat{G}_{k}=G\circ P_{k}$.}
	\label{Figure 1.}
\end{figure}

In several problems of graph theory, constructing a DGQS graph seem to be more difficult than determining whether a graph to be DGQS. Tian, Wu, Cui and Sun constructed many $DGQS$ graphs by using the rooted product graphs and proved the following theorems in \cite{tw},
\begin{thm}
	Let $G$ be a graph with ${\rm det}W_{Q}(G)=\pm 2^{\frac{3n-2}{2} } $ ($n$ is even) and the constant
	term $a_{0}$ with respect to the characteristic polynomial of G is equal to $\pm2$. Then every graphs in the infinite seqeuences $G\circ P_{k}^{t}$ ($k=2,3;~t \ge 1$) are DGQS.
\end{thm}
In this paper, we will prove the conclusion about $G\circ P_{k}^{t}$ is true for any positive integer $k$ and our paper is organzied as follow. In Section 2, we give some basic notations and lemmas. In Setion 3, we will give our main results and their proofs.

\section{Preliminaries}
In this section, we present some notations and lemmas that will be used in the following sections. Firstly,  we give three sequences $a_{k-1}(t)$,  $b_{k}(t)$ and $c_k(t)$ respectively, which satisfy the recursive relations of the following polynomials:
\begin{equation*}\label{equation1}
	a_{0}(t)=1,\quad a_{1}(t)=t-1 \quad {\rm and} \quad a_{k-1}(t)=(t-2)a_{k-2}(t)-a_{k-3}(t) \quad {\rm for} \quad k\ge3,
\end{equation*}
\begin{equation*}
	b_{0}(t)=1,\quad b_{1}(t)=t-1 \quad {\rm and} \quad b_{k}(t)=(t-1)a_{k-1}(t)-a_{k-2}(t) \quad {\rm for} \quad k\ge2.
\end{equation*}
\begin{equation*}
	c_{0}(t)=1,\quad c_{1}(t)=t-2 \quad {\rm and} \quad c_{k}(t)=(t-1)c_{k-1}(t)-c_{k-2}(t) \quad {\rm for} \quad k\ge2.
\end{equation*}
Besides,  we denote $f_{k}(t)=a_0(t)+a_1(t)+\cdots+a_{k-1}(t)$.

Our paper use the following equality frequently. Suppose $f(t)=\prod_{1\le i\le n}(t-\alpha_i)$, $g(t)=\prod_{1\le j\le m}(t-\beta_j)$,
then
\begin{align*}
	\prod_{1\le j\le m}f(\beta_j)&=\prod_{1\le j\le m}\prod_{1\le i\le n}(\beta_j-\alpha_i)\\
	&=(-1)^{mn}\prod_{1\le j\le m}\prod_{1\le i\le n}(\alpha_i-\beta_j)\\
	&=(-1)^{mn}\prod_{1\le i\le n}g(\alpha_i) \tag{*}
\end{align*}
In particular, if $n=m+1$, then $\prod_{1\le j\le m}f(\beta_j)==\prod_{1\le i\le n}g(\alpha_i)$.

Let 
\[
C=
\begin{bmatrix}
	1&2-t_{i}^{(j)}&1&0&\cdots&0&0\\
	0&1&2-t_{i}^{(j)}&1&\cdots&0&0\\
	0&0&1&2-t_{i}^{(j)}&\cdots&0&0\\
	\vdots&\vdots&\vdots&\ddots&\ddots&\ddots&\vdots\\
	0&0&0&0&\cdots&2-t_{i}^{(j)}&1\\
	0&0&0&0&\cdots&1&1-t_{i}^{(j)}
\end{bmatrix}
\]
\[
D=
\begin{bmatrix}
	1&0&0&0&\cdots&0&0&a_{k-1}(t_{i}^{(j)})\\
	0&1&0&0&\cdots&0&0&a_{k-2}(t_{i}^{(j)})\\
	0&0&1&0&\cdots&0&0&a_{k-3}(t_{i}^{(j)})\\
	\vdots&\vdots&\vdots&\ddots&\ddots&\ddots&\vdots&\vdots\\
	0&0&0&0&\cdots&0&0&a_{3}(t_{i}^{(j)})\\
	0&0&0&0&\cdots&1&0&a_{2}(t_{i}^{(j)})\\
	0&0&0&0&\cdots&0&1&a_{1}(t_{i}^{(j)})\\
\end{bmatrix}
\]
\begin{lem}[\cite{LHH}]\label{CD}
	If $\textbf{x}$ is a column vector in $\mathbb{R}^n$, then $C\textbf{x}=0$ if and only if $D\textbf{x}=0$.
\end{lem}

\begin{lem}\label{maintool2}
	Let $\lambda _{i} $ be the eigenvalue of the graph $G$ with respect to the normalized eigenvector $x_{i}$ ($i=1,2,\dots,n$). Let $t_{i}^{(j)} $ ($i=1,2,\dots,n;j=1,2,\dots,k$  ) be the $Q$-eigenvalues of the rooted product graph $\widehat{G}_{k}$ if and only if $t_{i}^{(j)}$ satisfies the following equation: $\varphi(t_{i}^{(j)})=b_{k}(t_{i}^{(j)})-\lambda_{i} a_{k-1}(t_{i}^{(j)})$.
\end{lem}

\begin{proof}
	The sufficient condition have been already proved in the proof of \cite{LHH}[lemma 3.1].
	It suffices to show that the necessary condition holds.
	Suppose that $t_{i}^{(j)}$ satisfies 
	$b_{k}(t_{i}^{(j)} ) - \lambda_{i} a_{k-1}(t_{i}^{(j)} )=0.$ for the eigenvalue $\lambda_i$ of $Q$.
	
	\noindent\emph{Claim} 1: $a_{k-1}(t_{i}^{(j)} )\neq 0$. 
	
	The proof of the claim is by contradiction. Since otherwise,
	we have $\lambda_{i} a_{k-1}(t_{i}^{(j)} )=0$, and by $\varphi(t_{i}^{(j)} )=b_{k}(t_{i}^{(j)} ) - \lambda_{i} a_{k-1}(t_{i}^{(j)} )=0$, we have $a_k(t_{i}^{(j)})=0$. It follwos that $a_{k-2}(t_{i}^{(j)})=(t_{i}^{(j)}-2)a_{k-1}(t_{i}^{(j)})-a_{k}(t_{i}^{(j)})=0$. Repeating thie process, 
	we finally get $a_2(t_{i}^{(j)})=a_1(t_{i}^{(j)})=0$, so $a_0(t_{i}^{(j)})=0$. But $a_0(t)=1$, this leads to a contradiction.
	we prove the claim.
	
	Then we set $x_1$ the corresponding eigenvector of $\lambda_i$. Let $x_i=\dfrac{a_{s-i}(t_{i}^{(j)})}{a_{s-1}(t_{i}^{(j)})}x_1$, $2\le i\le s$ and $\x=(x_1^T,x_2^T,\cdots,x_s^T)^T$. \\
	\noindent\emph{Claim} 2: $\x$ is an eigenvector corresponding to $t_{i}^{(j)}$ of the rooted product graph $\widehat{G}_{k}$.
	
	We easily find that  $Q_s\x=t_{i}^{(j)}\x$ is equivalent to 
	$$Q(G)x_1+x_2=(t_{i}^{(j)}-1)x_1$$
	\begin{equation}\label{eq23}
		\begin{cases}
			&x_1+x_3=(t_{i}^{(j)}-2)x_2\\
			&x_2+x_4=(t_{i}^{(j)}-2)x_3\\
			&\vdots\\
			&x_{s-3}+x_{s-1}=(t_{i}^{(j)}-2)x_{s-2}\\
			&x_{s-2}+x_{s}=(t_{i}^{(j)}-2)x_{s-1}\\
			&x_{s-1}=(t_{i}^{(j)}-1)x_s\\
		\end{cases}
	\end{equation}
	The equalities \ref{eq23} are also equivalent to $C\x=0$. and by lemma \ref{CD}, they are equivalent to $D\x=0$. This is obviously implied by the formula of $x_i$ and
	$$Q(G)x_1+x_2=\lambda_ix_1+x_2=\biggl(\dfrac{b_{s}(t_{i}^{(j)})+a_{s-2}(t_{i}^{(j)})}{a_{s-1}(t_{i}^{(j)})}\biggr)x_1
	=(t_{i}^{(j)}-1)x_1.$$
	We have thus proved the lemma by the claim.
	
\end{proof}

To prove the main result, we need to study some properties of $f_{k}(t)$ and $a_k(t)$, $b_k(t)$.
\begin{lem}\label{ap}
	Let $\lambda _{i} $, $t_{i}^{(j)}$ be as lemma \ref{maintool}. Then for any $1\le i\le n$,
	\begin{equation}
		\prod_{1\le k_{2}\le k }a_{k-1}(t_{i}^{(k_{2} )} ) =(-1)^{\frac{k(k-1)}{2}}.
	\end{equation}
\end{lem}
\begin{proof}
	According to lemma \ref{maintool2}, $t_{i}^{(j)}$ are all roots of the equation
	$\varphi(t)=b_{k}(t) - \lambda_{i} a_{k-1}(t)$. So we can set 
	$\varphi(t)=\prod_{1\le k_{2}\le k }(t-t_{i}^{(j)})$ and  $\alpha_1,\cdots,\alpha_{k-1}$ be all roots of $a_{k-1}(t)$,
	then by (*), we have 
	$$\prod_{1\le k_{2}\le k }a_{k-1}(t_{i}^{(k_{2} )} )=\prod_{1\le j\le k-1 }\varphi(\alpha_j)$$
	
	Then we have $\varphi(\alpha_j)=b_{k}(\alpha_j ) - \lambda_{i} a_{k-1}(\alpha_j )=a_k(\alpha_j)=-a_{k-2}(\alpha_j)$.
	It suffices to show $ \prod_{1\le j\le k-1 }a_{k-2}(\alpha_j)=(-1)^{\frac{(i-2)(i-1)}{2}}$
	
	{\bf Claim 1}: For any $i$, if $\alpha_1^i,\cdots,\alpha_{i}^i$ are the roots of $a_i(t)$, then
	$$\prod_{1\le j\le i}a_{i-1}(\alpha_{j}^i)=(-1)^{\frac{i(i-1)}{2}}$$
	The claim is proved by induction on $i$.
	
	For $i=1$, we have $a_0(1)=1$. When $i=2$, $a_1(\frac{\sqrt{5}+3}{2})a_1(\frac{3-\sqrt{5}}{2})=-1$ \\
	We suppose when $i=n-1$, claim holds, i.e. $\prod_{1\le j\le i-1}a_{i-2}(\alpha_{j}^{i-1})=(-1)^{\frac{(i-2)(i-1)}{2}}$. \\
	When $i=n$, we have
	$$\prod_{1\le j\le i}a_{i-1}(\alpha_{j}^i)=\prod_{1\le j\le i-1}a_{i}(\alpha_{j}^{i-1})=(-1)^{i-1}\prod_{1\le j\le i-1}a_{i-2}(\alpha_{j}^{i-1})=(-1)^{\frac{i(i-1)}{2}}$$
	So by the above claim, we can get the desired result.
\end{proof}

\begin{lem}[\cite{LHH}]\label{fs}
	\[f_{s}(t)=\begin{cases}
		a_k^2(t)\enspace &if\enspace s=2k+1,\enspace k\geq 0\\
		tc_k(t) \enspace &if\enspace s=2k,\enspace k\geq 1\\
	\end{cases}\]
\end{lem}

\begin{lem}[\cite{LHH}]\label{cp}
	For $s\geq 2$, we have
	$$tc_{s-1}(t)=(t-1)a_{s-1}(t)-a_{s-2}(t)=b_{s}(t)$$
\end{lem}
\begin{lem}\label{zero}
	Let $b_{i}(t),f_{i}(t)$ be as in section 2. If all roots of $f_{i}$ are $\beta_1^i,\cdots,\beta_i^i$, then for any $1\le j\le i$, 
	$b_{i}(\beta_j)=0$ .
\end{lem}
\begin{proof}
	Firstly, we prove the following claim.\\
	{\bf claim}: $a_{k+i}(t)+a_{k-i}(t)=(a_i(t)-a_{i-1}(t))a_k(t),(0<i<k+1)$\\
	By induction, when $i=1$, 
	$$a_{k+1}(t)+a_{k-1}(t)=(t-2)a_k(t)=(a_1(t)-a_0(t))a_k(t).$$
	We suppose when $i=n-1$, claim holds.\\
	When $i=n$, we have
	\begin{align*}
		a_{k+n}(t)+a_{k-n}(t)&=(t-2)(a_{k+n-1}(t)+a_{k-n+1}(t))-(a_{k+n-2}(t)+a_{k-n+2}(t))\\
		&=(t-2)(a_{n-1}(t)-a_{n-2}(t))a_k(t)-(a_{n-2}(t)-a_{n-3}(t))a_k(t)\\
		&=(a_n(t)-a_{n-1}(t))a_k(t).
	\end{align*}
	So we have proved the claim by induction $i$.\\
	Now we begin to prove the lemma
	We define $a_{-1}=-1$ and easily get $a_{1}(t)=(t-2)a_0(t)-a_{-1}(t)$. $\{a_{-1}(t),a_0(t),\cdots,a_k(t),\cdots,\}$ is also the sequence
	that satisfies the iterative relationship $a_k(t)=(t-2)a_{k-1}(t)-a_{k-2}(t)$ and the claim is also 
	obtained.\\
	So when $i=2k+1$, by the above cliam, we have
	\begin{align*}
		b_{2k+1}(t)&=a_{2k+1}(t)+a_{2k}(t)\\
		&=(a_{2k+1}(t)+a_{-1}(t))+(a_{2k}(t)+a_0(t))\\
		&=(a_{k+1}(t)-a_{k-1}(t))a_k(t). 
	\end{align*}
	So we can get $a_k(t)|b_{2k+1}(t)$. Combining with lemma \ref{fs}, when $i=2k+1$, the conclusion is obtained.
	
	When $i=2k$, we easily find that $f_{k}$ and $tc_k(t)$ have the same roots. Then by lemma \ref{cp},
	$b_k(t)=tc_k(t)$. So we get the desired result by combining the above dicussion.
\end{proof}

\begin{lem}\label{fp}
	For any $k\geq 1$, let $\beta_j^k,(1\le j\le k)$ be all roots of $f_k(t)$. Then we have 
	$\prod_{1\le j\le k} a_{k-1}(\beta_j^k)=(-1)^k$
\end{lem}
\begin{proof}We prove the lemma by the induction on $i$.
	When $i=1$, $a_{0}(t)=1$. When $i=2$, the roots of $f_2(t)=t$ is $0$. So $\prod_{j} a_{1}(0)=-1$.
	
	Now suppose when $i=k-1$, the lemma holds, i.e. $\prod_{1\le j\le k-1} a_{k-2}(\beta_j^{k-1})=(-1)^{k-1}$.
	Then when $i=k$, we easily get
	\begin{align*}
		\prod_{1\le j\le k} a_{k-1}(\beta_j^k)&=(-1)^k\prod_{1\le j\le k}f_{k-1}(\beta_j^k)\\
		&=(-1)^k\prod_{1\le j\le k-1}f_{k}(\beta_j^{k-1})=(-1)^k\prod_{1\le j\le k-1}a_{k-1}(\beta_j^{k-1})
	\end{align*}
	On the other hand, by lemma \ref{zero}, we have
	$$0=b_{k-1}(\beta_j^{k-1})=a_{k-1}(\beta_j^{k-1})+a_{k-2}(\beta_j^{k-1})$$
	So $$-a_{k-1}(\beta_j^{k-1})=a_{k-2}(\beta_j^{k-1})$$
	Then
	$$\prod_{1\le j\le k} a_{k-1}(\beta_j^k)=(-1)^{2k-1}\prod_{1\le j\le k-1}a_{k-2}(\beta_j^{k-1})=(-1)^{k}$$
\end{proof}

\section{ main result}
\subsection{The determinant of $W_{Q}(\widehat{G}_{k})$}

\begin{lem}\label{maintool}
	Let $\lambda _{i} $ be the $Q$-eigenvalues of the graph $G$ with respect to the normalized eigenvectors $x_{i}$ ($i=1,2,\dots ,n$). Let $t_{i}^{(j)}$ be the eigenvalues of the rooted product graph $G\circ P_{k}$ with respect to the eigenvectors $\zeta _{i}^{j}$ ($i=1,2,\dots,n;~j=1,2,\dots ,k$). Then
	\begin{equation}\label{itool1}
		{\rm det}(W_{Q}(G\circ P_{k} ))=\pm({\rm det}(W_{Q}(G))^{k}\prod_{1\le i\le n} \prod_{1\le k_{2}\le k }a_{k-1}(t_{i}^{(k_{2} )} ) f_{k}(t_{i}^{(k_{2}) } ).
	\end{equation}
\end{lem}

\begin{thm}\label{main result}
	Let $G$ be a graph with $a_{0}$ is the constant term with respect to the characteristic polynomial of the matrix $Q(G)$. Then  for any $k\ge 1$, ${\rm det} W(G\circ P_{k} )=(-1)^{\frac{nk(k+1)}{2}}a_{0}^{k-1}({\rm det}(W(G)))^{k} $. 
\end{thm}

\begin{rem}
	During the writing of this article, the authors found that theorem \ref{main result} was proved by Wang et al. in \cite{ymw} using a different method.
\end{rem}
\begin{proof}[Proof of Theorem \ref{main result}]
	By lemma \ref{maintool},
	\begin{equation*}\label{itool1}
		{\rm det}(W_{Q}(G\circ P_{k} ))=\pm({\rm det}(W_{Q}(G))^{k}\prod_{1\le i\le n} \prod_{1\le k_{2}\le k }a_{k-1}(t_{i}^{(k_{2} )} ) f_{k}(t_{i}^{(k_{2}) } ).
	\end{equation*}
	Let $\beta_j^k,(1\le j\le k)$ be all roots of $f_k(t)$. Then by lemma \ref{ap}, lemma \ref{fp}, and $(*)$
	\begin{align*}
		&\prod_{1\le i\le n} \prod_{1\le k_{2}\le k }a_{k-1}(t_{i}^{(k_{2} )} ) f_k(t_{i}^{(k_{2}) } )\\
		&=(-1)^{\frac{nk(k-1)}{2}}\prod_{1\le i\le n} \prod_{1\le k_{2}\le k-1 }\varphi_{k}(\beta_j^k)\\
		&=(-1)^{\frac{nk(k-1)}{2}}\prod_{1\le i\le n}(-\lambda_{i})^{k-1}(-1)^{k}\\
		&=(-1)^{\frac{nk(k+1)}{2}}a_0^{k-1},
	\end{align*}
	which completes the proof of Theorem \ref{main result}.
\end{proof}

\subsection{Constructig a infinite families of $DGQS$ graphs}
Let $G$ be a graph and $a_{0}$ be the constant term with respect to the characteristic polynomial of the matrix $Q(G)$.
\begin{lem}\label{DGQS}
	Let $G$ be a graph. If $2^{-\left \lfloor \frac{3n-2}{2}  \right \rfloor }\det(W_Q(G))$ is odd and square-free, then $G$ is $DGQS$.
\end{lem}
\begin{lem}\label{222}
	Let $G$ be a graph with $a_{0}=\pm 2$. If $a_{0}^{(t)} $ is the constant term with respect to the characteristic polynomial of the matrix $Q(G\circ P_{k} ^{t})$, then $a_{0}^{(t)}=\pm 2$.
\end{lem}

\begin{thm}\label{main result2}
	Let $n$ be an even positive integer and $G$ be a graph with $n$ vectices. If ${\rm det}W_{Q}(G)=\pm 2^{\frac{3n-2}{2} }$and $a_{0}=\pm2$, then every graph in the infinite seqeuence $G\circ P_{k}^{t} $ ($k\ge 2;~t \ge 1$) are $DGQS$
\end{thm}
\begin{proof}[Proof of Theorem \ref{main result2}]
	By theorem \ref{main result}, lemma \ref{222}, we have
	$${\rm det}W_{Q} (G\circ P_{k}^{t}  )=\pm (2)^{k-1}({\rm det}W_{Q} (G\circ P_{k}^{t-1}  )) ^{k} =\cdots =\pm 2^{\frac{3k^tn}{2}-1}.$$
	Meanwhile, the number of the vertices in $G\circ P_{k} ^{t}$ is $k^tn$. Thus , we get that 
	$$\dfrac{{\rm det}W_{Q} (G\circ P_{k}^{t} )}{2^{\frac{3k^tn}{2}-1}}=\pm 1.$$	
	Therefore,  $G\circ P_{k} ^{t}$ are $DGQS$ from lemma \ref{DGQS}.
	
\end{proof}

\end{document}